\documentclass[runningheads,a4paper]{llncs}

\usepackage{amssymb,amsmath,wasysym}
\usepackage{graphicx,epsfig}
\usepackage{alltt}
\usepackage{framed}
\usepackage{dsfont}
\usepackage{url}
\usepackage{hyperref}

\newcommand{\keywords}[1]{\par\addvspace\baselineskip
\noindent\keywordname\enspace\ignorespaces#1}

\let\phi\varphi
\let\epsilon\varepsilon

\newcommand{\bvp}[2]{\boxed{\begin{array}{l}#1\\#2\end{array}}}
\newcommand{\double}[2]{\begin{array}{l}#1\\#2\end{array}}

\newcommand{\ZZ}{\mathbb{Z}}

\newcommand{\CC}{\mathbb{C}}

\newcommand{\der}{\partial}

\newcommand{\cum}{{\textstyle \varint}}

\newcommand{\galg}{\mathcal{F}}

\newcommand{\fri}[1]{#1^\lozenge}

\DeclareMathOperator{\Ker}{Ker}
\DeclareMathOperator{\Img}{Im}

\newcommand{\tma}{TH$\exists$OREM$\forall$}
\newcommand{\grgr}{\texttt{GreenGroebner}}

\newcommand{\fafac}[2]{#1^{\underline{#2}}}

\renewcommand{\binom}[2]{\begin{pmatrix}#1\\#2\end{pmatrix}}

\newcommand{\evl}{\text{\scshape\texttt e}}

\newcommand{\ocum}{{\setbox0=%
    \hbox{$\textstyle{\scriptstyle-}{\varint}$}%
    \textstyle{\vcenter{\hbox{$\scriptstyle-$}}\kern-.5\wd0}%
    \!\varint}}

\newcommand{\intdiffop}{\galg[\der,\cum]}

\newcommand{\inner}[2]{\langle #1 | #2 \rangle}

\newcommand{\dirs}{\dotplus}

\newcommand{\mer}[1]{\mathcal{M}(#1)}
\newcommand{\pp}{\mathrm{pp}}
\newcommand{\reg}{\mathrm{reg}}
\newcommand{\orth}[1]{#1^\perp}
\newcommand{\bspc}{\mathcal{B}}
\newcommand{\espc}{\mathcal{E}}
\newcommand{\aspc}{\mathcal{A}}
\newcommand{\ospc}{\mathcal{C}}
\newcommand{\coeffex}[1]{\langle x^{#1}\rangle}
\newcommand{\indproj}[1]{\langle #1 \rangle}

\providecommand{\abs}[1]{\lvert#1\rvert}

\spnewtheorem{myexample}[theorem]{Example}{\bfseries}{\upshape}
\spnewtheorem{mydefinition}[theorem]{Definition}{\bfseries}{\upshape}
\spnewtheorem{mylemma}[theorem]{Lemma}{\bfseries}{\itshape}
\spnewtheorem{myproposition}[theorem]{Proposition}{\bfseries}{\itshape}

\begin{document}

\mainmatter
\title{Two-Point Boundary Problems with\\ One Mild Singularity and an
  Application to Graded Kirchhoff Plates}
\titlerunning{Two-Point Boundary Problems with Singularity}

\author{Markus Rosenkranz\inst{1} \and Jane Liu\inst{2} \and
  Alexander Maletzky\inst{3} \and Bruno Buchberger\inst{3}}
\authorrunning{M. Rosenkranz, J. Liu, A. Maletzky and B. Buchberger}

\institute{University of Kent, \\
Cornwallis Building, Canterbury, Kent CT2 7NF, United Kingdom
\and
Tennessee Tech University, United States
\and
Research Institute for Symbolic Computation (RISC),
Hagenberg, Austria
}

\maketitle

\begin{abstract}
  We develop a new theory for treating boundary problems for linear
  ordinary differential equations whose fundamental system may have a
  singularity at one of the two endpoints of the given interval. Our
  treatment follows an algebraic approach, with (partial)
  implementation in the Theorema software system (which is based on
  Mathematica). We study an application to graded Kirchhoff plates for
  illustrating a typical case of such boundary problems.
\end{abstract}


\section{Introduction}
\label{sec:intro}

The treatment of boundary problems in symbolic computation was
initiated in the PhD thesis~\cite{Rosenkranz2009} under the guidance
of Bruno Buchberger in cooperation with Heinz Engl; see
also~\cite{RosenkranzBuchbergerEngl2003,Rosenkranz2005} and for the
further
development~\cite{RosenkranzRegensburger2008a,RegensburgerRosenkranz2009,RosenkranzRegensburgerTecBuchberger2012}. Its
implementation was originally carried out within the \tma\
project~\cite{BuchbergerCraciunJebeleanEtAl2006}.

Up to now, we have always assumed differential equations without
singularity or, equivalently, monic differential operators (leading
coefficient function being unity). In this paper, we develop for the
first time an \emph{algebraic theory} for treating boundary problems
with a (mild) \emph{singularity at one endpoint}. For details, we
refer to Section~\ref{sec:bp-one-sing}. Our approach is very different
from the traditional analysis setting in terms of the Weyl-Titchmarsh
theory (limit points and limit circles). It would be very interesting
to explore the connections between our approach and the classical
treatment; however, this must be left for future work.

Regarding the general setup of the \emph{algebraic language for
  boundary problems}, we refer to the references mentioned above, in
particular~\cite{RosenkranzRegensburgerTecBuchberger2012}. At this
point, let us just recall some notation. We start from a fixed
integro-differential algebra~$(\galg, \der, \cum)$. The formulation of
(local) boundary conditions is in terms of \emph{evaluations}, which
are by definition multiplicative linear functionals in~$\galg^*$. We
write~$\galg_1 \le \galg_2$ if~$\galg_1$ is a \emph{subspace}
of~$\galg_2$; the same notation is used for subspaces of the dual
space~$\galg^*$. The \emph{orthogonal} of a subspace~$\bspc \le
\galg^*$ is defined as
\begin{equation*}
  \orth{\bspc} = \{ u \in \galg \mid \beta(u) = 0 \text{ for all
    $\beta \in \bspc$} \}
\end{equation*}
and similarly for the orthogonal of a subspace~$\aspc \le \galg$. We
write~$[f_1, f_2, \dots]$ for the \emph{linear span} of (possibly
infinitely many) elements~$f_1, f_2, \dots \in \galg$; the same
notation is used for linear spans within~$\galg^*$. The \emph{zero
  vector space} of any dimension is denoted by~$O$.

We write
\begin{equation*}
  \fafac{n}{k} = n(n-1) \cdots (n-k+1)
\end{equation*}
for the \emph{falling factorial}, where~$n \in \CC$ could be arbitrary
(but will be an integer for our purposes) while~$k$ is taken to be a
natural number. Note that~$\fafac{n}{k} = 0$ for~$k>n$. Our main
example of an integro-differential algebra will be the
space~$C^\omega(I)$ of analytic functions on a closed
interval~$I=[a,b]$. Recall that by definition this is the space of
functions that are analytic in some open set containing~$[a,b]$.

The development of the new algebraic theory of two-point boundary
problems with a mild singularity (whose treatment is really just
broached in this paper) was triggered by a collaboration between a
\emph{symbolic computation} team (consisting of the first, second and
fourth author) and a researcher in \emph{engineering mechancs} (the
third author). This underlines the importance and fruitfulness of
collaborations between theoretical developments and practical
applications. We present the example that had originally led to this
research in Section~\ref{sec:appl}.

From a methodological point of view, this research on the symbolic
algorithm for linear boundary value problems has particular relevance
and is drawing from the Theorema Project, see
\cite{BuchbergerCraciunJebeleanEtAl2006}. This project aims at
supporting a new paradigm for doing (algorithmic) mathematical
research: In the phase of doing research on new theorems and
algorithms, \tma\ provides a formal language (a version of
predicate logic) and an automated reasoning system by which the
exploration of the theory is supported. In the phase in which
algorithms based on the new theory should be implemented and used in
computing examples, \tma\ allows to program and execute the
algorithms \emph{in the same language}. In the particular case of our
approach to solving linear boundary value problems, the fundamental
theorem on which the approach is based was proved automatically by
checking that the rewrite rules for integro-differential operators
forms a Gr\"obner basis. In a second step, the algorithm for solving
linear boundary value problems is expressed again in the \tma\
language and can then be called by the users by inputting the linear
boundary value problems in a user-friendly notation.

In its current version, the engine for solving boundary problems is
bundled in the \grgr\ package of \tma. As an example, consider the
boundary problem
\[
  \bvp{u'' + \tfrac{1}{x} u' - \tfrac{1}{x^2} u = f,}{u(0)=u(1)=0,}
\]
which we shall consider in greater detail in
Section~\ref{sec:simple-example}. This can be given to \tma\ in the
form
\[
  \mathtt{BPSolve}\left[\ \bvp{\mathtt{u'' + \tfrac{1}{x}\,u' - \tfrac{1}{x^2}\,u = f}}{\mathtt{u[0]=u[1]=0}}\mathtt{, u, x, 0, 1}\right]
\]
leading  either to the solution for $u(x)$ as
\[
  \mathtt{-\frac{1}{2}\frac{1}{x}\left(x^2 \int_0^1 f[\xi]\,\mathbf{d}\xi - x^2 \int_0^x f[\xi]\,\mathbf{d}\xi - x^2 \int_0^1 \xi^2 f[\xi]\,\mathbf{d}\xi + \int_0^x \xi^2 f[\xi]\,\mathbf{d}\xi \right)}
\]
or to the Green's function $g(x,\xi)$ as
\[
  \begin{cases}
    \mathtt{\frac{1}{2}\frac{1}{x}\,(-1+x^2)\,\xi^2} & \mathtt{\Leftarrow\quad \xi\leq x}\\
    \mathtt{\frac{1}{2}\,x\,(-1+\xi^2)} & \mathtt{\Leftarrow\quad x < \xi}
  \end{cases}
\]
at the user's request.

In addition to the \grgr\ package in \tma, a Maple package named
\texttt{IntDiffOp} is also
available~\cite{KorporalRegensburgerRosenkranz2011}. This package was
developed in the frame of Anja Korporal's PhD thesis, supervised by
Georg Regensburger and the first author. The Maple package supports
also generalized boundary problems (see
Section~\ref{sec:simple-example} for their relevance to this paper).
One advantage of the \tma\ system is that both the research phase and
the application phase of our method can be formulated and supported
within the same logic and software system---which we consider to be
quite a novel and promising paradigm for the future.

\section{A Simple Example}
\label{sec:simple-example}

For illustrating the new ideas, it is illuminating to look at a simple
example that exhibits the kind of phenomena that we have to cope with
in the Kirchhoff plate boundary problem.

\begin{myexample}
  \label{ex:simple-example}
  Let us start with the \emph{intuitive but mathematically unprecise}
  statement of the following example: Given a ``suitable'' forcing
  function~$f$ on the unit interval~$I=[0,1]$, we want to find a
  ``reasonable'' solution function~$u$ such that
  \begin{equation}
    \label{eq:simple-bvp}
    \bvp{u'' + \tfrac{1}{x} u' - \tfrac{1}{x^2} u = f,}{u(0)=u(1)=0.}
  \end{equation}
  But note that the differential operator~$T = D^2 + \tfrac{1}{x} D -
  \tfrac{1}{x^2}$ is singular at the left boundary point~$x=0$ of the
  interval~$I$ under consideration. Hence the first boundary
  condition~$u(0)=0$ should be looked at with some suspicion. And what
  function space are we supposed to consider in the first place? If
  the~$\tfrac{1}{x}$ and~$\tfrac{1}{x^2}$ are to be taken literally,
  the space~$C^\infty[0,1]$ will clearly not do. On the other hand, we
  need functions that are smooth (or at least continuous) at~$x=0$ to
  make sense of~$u(0)$.
\end{myexample}

In the rest of this section, we shall give one possible solution to
the dilemma outlined above. Of course we could resort to using
different function space for~$u$ and~$f$, and this is in fact the
approach one usually takes in Analysis. For our present purposes,
however, we prefer to keep the simple paradigm of integro-differential
algebras as outlined in Section~\ref{sec:intro}, but we shall modify
it to accommodate singularities such as in~$\tfrac{1}{x}$
and~$\tfrac{1}{x^2}$. Note that these are just poles, so we can
take~$\galg$ to be the subring of the field~$\mer{I}$ consisting of
complex-valued \emph{meromorphic functions} that are regular at all~$x
\in I$ except possibly~$x=0$. In other words, these are functions that
have a Laurent expansion at~$x=0$ with finite principal part,
converging in a complex annulus~$0 < \abs{z} < \rho$ with~$\rho >
1$. In fact, we will only use the real part~$[-1,1] \setminus \{0\}$
of this annulus.  Note that we have of course~$\tfrac{1}{x},
\tfrac{1}{x^2} \in \galg$.

The ring~$\galg$ is an integro-differential algebra over~$\CC$ if we
use the standard \emph{derivation}~$\der = \tfrac{d}{dx}$ and the
\emph{Rota-Baxter operator}
\begin{equation*}
  \cum f := \cum_1^x f(\xi) \, d\xi
\end{equation*}
initialized at the regular point~$x=1$. 

We have now ensured that the differential operator
of~\eqref{eq:simple-bvp} has a clear algebraic interpretation~$T \in
\galg[\der]$. However, the boundary condition~$u(0)=0$ is still
dubious. For making it precise, note that the integro-differential
algebra~$(\galg, \der, \cum)$ has only the second of the two boundary
evaluations~$L, R\colon \galg \to \CC$ with~$L(f) = f(0)$ and~$R(f) =
f(1)$ in the usual sense of a total function. So while we can
interpret the second boundary condition algebraically by~$Ru = 0$, the
same does not work on the left endpoint. Instead of an evaluation
at~$x=0$ we shall introduce the map
\begin{equation*}
  \pp\colon \sum_{n=N}^\infty a_n x^n \mapsto \sum_{n=N}^{-1} a_n x^n
\end{equation*}
that extracts the \emph{principal part} of a function written in terms
of its Laurent expansion at~$x=0$. Here and henceforth we assume such
expansions of nonzero functions are written with~$a_N \neq
0$. If~$N\ge0$ the function is regular at~$x=0$, and the above sum is
to be understood as zero. Clearly, $\pp\colon \galg \to \galg$ is a
linear projector, with the complementary projector~$\reg := 1_\galg -
\pp$ extracting the \emph{regular part} at~$x=0$. Incidentially, $\pp$
and $\reg$ are also Rota-Baxter operators of weight~$-1$, which play
a crucial role in the renormalization theory of perturbative quantum
field theory~\cite[Ex.~1.1.10]{Guo2012}.

Finally, we define the functional~$C\colon \galg \to \CC$ that
extracts the \emph{constant term}~$a_0$ of a meromorphic function
expanded at~$x=0$. Combining~$C$ with the monomial multiplication
operators, we obtain the coefficient functionals~$\coeffex{n} := C
x^{-n} \: (n \in \ZZ)$ that map~$\sum_n a_n x^n$ to~$a_n$. In
particular, the residue functional is given by~$\coeffex{-1} = Cx$.

Note that for functions regular at~$x=0$, the functional~$C$ coincides
with the evaluation at the left endpoint, $L\colon \galg \to \CC, f
\mapsto f(0)$. However, for general meromorphic functions, $L$ is
undefined and $C$ is \emph{not multiplicative} since for example~$C(x
\cdot \tfrac{1}{x}) = 1 \neq 0 = C(x) \, C(\tfrac{1}{x})$. Hence we
refer to~$C$ only as a functional but not as an evaluation.

We can now make the boundary condition~$u(0) = 0$ precise for our
setting over~$\galg$. What we really mean is that~$\lambda(u) = 0$,
where~$\lambda := \pp + C$ is the projector that extracts the
principal part together with the constant term. Extending the
algebraic notion of boundary problem to allow for boundary conditions
that are not functionals, we may thus view~\eqref{eq:simple-bvp}
as~$(D^2 + \tfrac{1}{x} D - \tfrac{1}{x^2}, \orth{[\lambda, R]})$. In
this way we have given a precise meaning to the \emph{formulation of
  the boundary problem}.

But how are we to go about its \emph{solution} in a systematic manner?
Let us first look at the adhoc \emph{standard method} way of doing
this---determining the general homogeneous solution, then add the
inhomogeneous solution via variation of constants, and finally adapt
the integration constants to accommodate the boundary conditions. In
our case, one sees immediately that~$\Ker(T) = [x, \tfrac{1}{x}]$ so
that the general solution of the homogeneous differential equation
is~$u(x) = c_1 x + \tfrac{c_2}{x}$, where~$c_1, c_2 \in \CC$ are
integration constants. Variation of
constants~\cite[p.~74]{CoddingtonLevinson1955} then yields
\begin{equation}
  \label{eq:var-const}
  u(x) = c_1 x + \frac{c_2}{x} + \int_1^x \big(\frac{x}{2} -
  \frac{\xi^2}{2x}\big) \, f(\xi) \, d\xi
\end{equation}
for the inhomogeneous solution. Note that~$f \in \galg$ may also have
singularities as~$x=0$ or any other point~$x \in I$ apart from~$x=1$.

Now we need to impose the \emph{boundary conditions}. From~$u(1) = 0$
we obtain immediately~$c_1 + c_2 = 0$. For the boundary condition
at~$x=0$ we have to proceed a bit more cautiously, obtaining
\begin{align*}
  u(0) &= \lim_{x \to 0} \Big( c_1 x + \frac{c_2}{x} + \frac{x}{2}
  \int_1^x f(\xi) \, d\xi - \frac{1}{2x} \, \int_1^x
  \xi^2 f(\xi) \, d\xi \Big)\\
  &= \lim_{x \to 0} \Big( \frac{c_2}{x} - \frac{1}{2x} \, \int_1^x
  \xi^2 f(\xi) \, d\xi \Big)
\end{align*}
where we assume that~$f$ is regular at~$0$. It is clear that the
remaining limit can only exist if the integral tends to a finite limit
as~$x \to 0$, and this is the case by our assumption on~$f$. But then
we may apply the boundary condition to obtain~$c_2 = \tfrac{1}{2} \,
\cum_{\!1}^0 \, \xi^2 f(\xi) \, d\xi = - \tfrac{1}{2} \, \cum_{\!0}^1
\, \xi^2 f(\xi) \, d\xi$. This gives the overall solution
\begin{equation}
  \label{eq:simple-sol}
  u(x) = \bigg( \frac{x}{2} \int_0^1 \xi^2  - \frac{1}{2x}
  \int_0^1 \xi^2 {} - \frac{x}{2} \int_x^1 \! + \, \frac{1}{2x} \int_x^1
  \xi^2 \bigg) f(\xi) \, d\xi,
\end{equation}
which one may write in the standard form~$u(x) = \cum_0^1 \, g(x,\xi) \,
f(\xi) \, d\xi$ where the \emph{Green's function} is defined as
\begin{equation}
  \label{eq:simple-gf}
  g(x,\xi) = 
  \begin{cases}
    \tfrac{x\xi^2}{2} - \tfrac{\xi^2}{2x} & \text{if $\xi \le x$}\\
    \tfrac{x\xi^2}{2} - \tfrac{x}{2}      & \text{if $\xi \ge x$}
  \end{cases}
\end{equation}
in the usual manner.

How are we to make sense of this in an algebraic way, i.e.\@ without
(explicit) use of limits and hence topology? The key to this lies in
the projector~$\pp$ and the functional~$L$, which serve to distill
into our algebraic setting what we need from the topology
(namely~$f(x) = (\pp \, f)(x) + O(1)$ as $x \to 0$). However, there is
another complication when compared to boundary problems without
singularities as presented in Section~\ref{sec:intro}: We cannot
expect a solution~$u \in \galg$ to~\eqref{eq:simple-bvp} for every
given forcing function~$f \in \galg$. In other words, this boundary
problem is \emph{not regular} in the sense
of~\cite{RosenkranzRegensburger2008a}.

In the past, we have also used the term \emph{singular boundary
  problem} for such situations (here this seems to be suitable in a
double sense but we shall be careful to separate the second sense by
sticking to the designation ``boundary problems with
singularities''). The theory of singular boundary problems was
developed in an abstract setting in~\cite{Korporal2012}; applications
to boundary problems (without singularities) have been presented
in~\cite{KorporalRegensburgerRosenkranz2011}. At this point we shall
only recall a few basic facts.

A boundary problem~$(T, \bspc)$ is called \emph{semi-regular}
if~$\Ker(T) \cap \orth{\bspc} = O$. It is easy to see that the
boundary problem~$(D^2 + \tfrac{1}{x} D - \tfrac{1}{x^2},
\orth{[\lambda, R]})$ is in fact semi-regular. Since any~$u \in
\Ker(T)$ can be written as~$u(x) = c_1 x + \tfrac{c_2}{x}$, the
condition~$Ru = 0$ implies~$c_2 = -c_1$ and hence~$u(x) = c_1
(x-\tfrac{1}{x})$. But then~$(\lambda u)(x) = -\tfrac{c_1}{x} = 0$
forces~$c_1 = 0$ and hence~$u=0$.

If~$(D^2 + \tfrac{1}{x} D - \tfrac{1}{x^2}, \orth{[\lambda, R]})$ were
a regular boundary problem, we would have~$\Ker(T) \dirs
\orth{[\lambda, R]} = \galg$. However, it is easy to see that there
are elements~$u \in \galg$ that do not belong to~$\Ker(T) +
\orth{[\lambda, R]}$, for example~$u(x) = \tfrac{1}{x^2}$. Hence we
conclude that the boundary problem~\eqref{eq:simple-bvp} is in fact
overdetermined. For such boundary problems~$(T, \bspc)$ one can always
select a \emph{regular subproblem}~$(T, \tilde{\bspc})$, in the sense
that~$\tilde{\bspc} < \bspc$. In our case, a natural choice
is~$\tilde{\bspc} = [\coeffex{-1}, R]$. This is regular since the
evaluation matrix
\begin{equation*}
  (\coeffex{-1}, R)(\tfrac{1}{x}, x) = 
  \begin{pmatrix}
    1 & 0\\
    1 & 1
  \end{pmatrix}
\end{equation*}
is regular, and the associated kernel projector is~$P = \tfrac{1}{x}
\, \coeffex{-1} + x \, (R - \coeffex{-1})$
by~\cite[Lem.~A.1]{RegensburgerRosenkranz2009}.

For making the boundary problem~\eqref{eq:simple-bvp} well-defined
on~$\galg$ we need one more ingredient: We have to fix a
complement~$\espc$ of~$T(\orth{\bspc})$, the so-called
\emph{exceptional space}. Intuitively speaking, this comprises the
``exceptional functions'' of~$\galg$ that we decide to discard in
order to render~\eqref{eq:simple-bvp} solvable. Let us first work out
what~$T(\orth{\bspc}) \le \galg$ looks like. Since~$u(x) = \sum_{n\ge
  N}^\infty a_n x^n \in \orth{\bspc}$ forces the principal part and
constant term to vanish, we may start from~$u(x) = \sum_{n > 0} a_n
x^n$, with the additional proviso that~$\sum_{n>0} a_n = 1$.
Applying~$T = D^2 + \tfrac{1}{x} D - \tfrac{1}{x^2}$ to this~$u(x)$
yields~$\sum_{n>1} a_n \, (n^2 - 1) \, x^{n-2}$, which represents an
arbitrary element of~$C^\omega(I) = \orth{[\pp]}$ for a suitable
choice of coefficients~$(a_n)_{n>1}$ since the additional
condition~$\sum_{n>0} a_n = 1$ can always be met by choosing~$a_1 = 1
- \sum_{n>1} a_n$. But then it is very natural to choose~$\espc =
\orth{[\reg]}$ as the required complement. Clearly, the elements of
this space~$\espc$ are the Laurent polynomials of~$\CC[\tfrac{1}{x}]$
without constant term.

By Prop.~2 of~\cite{KorporalRegensburgerRosenkranz2011}, the
\emph{Green's operator} of a generalized boundary problem~$(T, \bspc,
\espc)$ is given by~$G = \tilde{G} Q$, where~$\tilde{G}$ is the
Green's operator of some regular subproblem~$(T, \tilde{\bspc})$
and~$Q$ is the projector onto~$T(\orth{\bspc})$ along~$\espc$. In our
case, the latter projector is clearly~$Q = \reg$, while the Green's
operator~$\tilde{G} = (1-P) \fri{T}$
by~\cite[Thm.~26]{RosenkranzRegensburger2008a}, with the kernel
projector~$P$ as above and the fundamental right inverse~$\fri{T}$
given according to~\cite[Prop.~23]{RosenkranzRegensburger2008a} by
\begin{equation*}
  \fri{T} = \tfrac{1}{2} \, A_1 x - \tfrac{1}{2x} \, A_1 x^2,
\end{equation*}
which is essentially just a reformulation of the inhomogeneous part
in~\eqref{eq:var-const}. Following the style of~\cite{Rosenkranz2005}
we write here~$\cum_1^x$ as~$A_1 \in \intdiffop$ to emphasize its role
in the integro-differential operator ring. Similarly, we write~$F :=
-LA_1 = \cum_0^1$ for the definite integral over the full
interval~$I$, which we may regard as a linear functional~$C^\omega(I)
\to \CC$.

Putting things together, it remains to compute
\begin{align*}
  G &= (1-P) \, \fri{T} Q = \big(1-\tfrac{1}{x} \, \coeffex{-1} + x \,
  \coeffex{-1} - x R\big) (\tfrac{1}{2} \, A_1 x - \tfrac{1}{2x} \, A_1
  x^2) \, \reg\\
  &= \big( -\tfrac{x}{2} A_1 + \tfrac{1}{2x} \, A_1 x^2 - \tfrac{1}{2x} \, F
  x^2 + \tfrac{x}{2} \, F x^2 \big) \reg,
\end{align*}
which may be done by the usual \emph{rewrite
  rules}~\cite[Tbl.~1]{RosenkranzRegensburger2008a} for the operator
ring~$\intdiffop$, together with the obvious extra rules that
on~$\Img(\reg) = C^\omega(I)$ the residual~$\coeffex{-1}$ vanishes
and~$C = \coeffex{0}$ coincides with the evaluation~$L$ at zero. Using
the standard procedure for extracting the \emph{Green's function}, one
obtains exactly~\eqref{eq:simple-gf} when restricting the forcing
functions to~$f \in C^\omega(I)$. We have thus succeeded in applying
the algebraic machinery to regain the solution previoulsy determined
by analysis techniques. More than that: The precise form of accessible
forcing functions is now fully settled, whereas the regularity
assumption in~\eqref{eq:simple-sol} was left somewhat vague (a
sufficient condition whose necessity was left open).

\section{Two-Point Boundary Problems with One Singularity}
\label{sec:bp-one-sing}

Let us now address the general question of specifying and solving
boundary problems (as usual: relative to a given fundamental system)
that have only \emph{one singularity}. The case of multipliple
singularities is left for future investigations. Using a scaling
transformation (and possibly a reflection), we may thus assume the
same setting as in Section~\ref{sec:simple-example}, with the
\emph{singularity at the origin} and the other boundary point at~$1$.

For the scope of this paper, we shall also restrict ourselves to a
certain subclass of Stieltjes boundary problems~$(T,\bspc)$: First of
all, we shall allow \emph{only local boundary conditions}
in~$\bspc$. This means multi-point conditions and higher-order
derivatives (leading to ill-posed boundary problems with
distributional Green's functions) are still allowed, but no global
parts (integrals); for details we refer
to~\cite[Def.~1]{RosenkranzSerwa2015}. The second restriction concerns
the differential operator~$T \in \CC(x)[\der]$, which we require to be
\emph{Fuchsian without resonances}. The latter means the differential
equation~$Tu=0$ is of Fuchsian type (the singularity is regular), and
has fundamental solutions~$x^{\lambda_1} \phi_1(x), \dots,
x^{\lambda_n} \phi_n(x)$ with each~$\phi_i \in C^\omega(I)$ having
order~$0$, where~$\lambda_1, \dots, \lambda_n$ are the roots of the
indicial equation~\cite[p.~127]{CoddingtonLevinson1955}. In other
words, we do not require logarithms for the solutions. (A
sufficient---but not necessary---condition for this is that the
roots~$\lambda_i$ are all distinct and do not differ by integers.)

\begin{mydefinition}
  We call~$(T, \bspc)$ a \emph{boundary problem with mild singularity}
  if~$T \in \CC(x)[\der]$ is a nonresonant Fuchsian operator
  and~$\bspc$ a local boundary space.
\end{mydefinition}

For a fixed~$(T, \bspc)$, we shall then enlarge the \emph{function
  space}~$\galg$ of Section~\ref{sec:simple-example} by
adding~$x^{\mu_1}, \dots, x^{\mu_n}$ as algebra generators, where
each~$\mu_i$ is the fractional part of the corresponding indicial
root~$\lambda_i$. Every element of~$\galg$ is then a sum of
series~$x^{\mu} \sum_{n \ge N} a_n x^n$, with~$\mu \in \{\mu_1,
\dots, \mu_n \}$ and~$N \in \ZZ$. The integro-differential structure
on~$(\galg, \der, \cum)$ is determined by setting~$\der x^{\mu} =
\lambda_i \, x^{\mu-1}$, as usual, and by using for~$\cum$ the
integral~$\cum_1^x$ as we did also in
Section~\ref{sec:simple-example}.

As \emph{boundary functionals} in~$\bspc$, we admit
derivatives~$\evl_\xi D^l\; (0 < \xi \le 1, l \ge 0)$ and coefficient
functionals~$\coeffex{k+\mu} \; (k \in \ZZ)$ whose action is~$x^{\mu}
\sum_{n \ge N} a_n x^n \mapsto a_k$. For functions~$f \in C^\omega(I)$
we have of course~$\coeffex{k+\mu} \, x^\mu f = f^{(k)}(0)/k!$. Since
the projectors~$\reg$ and~$\pp$ can be expressed in terms of the
coefficient functionals, we shall henceforth regard the latter just as
convenient abbrevations; boundary spaces are always written in terms
of~$\evl_\xi$ and~$\coeffex{k}$ only but of course they can be
infinite-dimensional. For example, in Section~\ref{sec:simple-example}
we had the ``regularized boundary condition'' $\lambda(u) = 0$, which
is equivalent to~$Lu = 0$ and~$\pp(u) = 0$ and hence to~$\coeffex{k}
\, u = 0 \; (k \le 0)$. Its full boundary space is therefore~$\bspc =
[R, \coeffex{k} \mid k \le 0]$.

The first issue that we must now address is the choice of
\emph{suitable boundary conditions}: Unlike in the ``smooth case''
(without singularities), we may not be able to impose $n$~boundary
conditions for an $n$-th order differential equation. The motivating
example of Section~\ref{sec:simple-example} was chosen to be
reasonably similar to the smooth case, so the presence of a
singularity was only seen in replacing the boundary evaluation~$L$ by
its regularized version~$\lambda$. As explained above, we were
effectively adding the extra condition~$\pp(u) = 0$ to the standard
boundary conditions~$u(0) = u(1) = 0$. In other cases, this will not
do as the following simple example shows.

\begin{myexample}
  \label{ex:ivp}
  Consider the nonresonant Fuchsian differential equation~$Tu(x) :=
  u'' + \tfrac{4}{x} \, u' + \tfrac{2}{x^2} \, u = 0$. Note that here
  the indicial equation has the roots~$\lambda_1 = -2$ and~$\lambda_2
  = -1$, which differ by the integer~$1$. Nevertheless, we may
  take~$\{ \tfrac{1}{x}, \tfrac{1}{x^2} \}$ as a fundamental system,
  so~$T$ is indeed nonresonant.

  Trying to impose the same (regularized) boundary
  space~$\tilde{\bspc} = [\pp, L, R]$ as in
  Example~\ref{ex:simple-example}, one
  obtains~$$T(\orth{\tilde{\bspc}}) = \Big\{ \sum_{n=-1}^\infty b_n
  x^n \mid \sum_{n=-1}^\infty \tfrac{b_n}{(n+3)(n+4)} = 0 \Big\}$$
  after a short calculation. But this means that the forcing
  functions~$f$ in the boundary problem
  \begin{equation*}
    \bvp{u'' + \tfrac{4}{x} \, u' + \tfrac{2}{x^2} \, u = f}{%
      \pp(u) = u(0) = u(1) = 0}
  \end{equation*}
  must satisfy an awkward extra condition (viz.\@ the one on the
  right-hand side of~$T(\orth{\tilde{\bspc}})$ above). This is not
  compensated by the slightly enlarged generality of allowing~$f$ to
  have a simple pole at~$x=0$.

  In the present case, we could instead impose \emph{initial
    conditions} at~$0$ so that~$\tilde{\bspc} = [\pp, L, LD]$. In this
  case one gets~$T(\orth{\tilde{\bspc}}) = C^\omega(I)$, so there is a
  unique solution for every analytic forcing function.
\end{myexample}

For a given nonresonant Fuchsian operator~$T \in \CC(x)[\der]$, a
better approach appears to be the following (we will make this more
precise below):
\begin{enumerate}
\item We compute first some boundary functionals~$\beta_1, \dots,
  \beta_n$ that ensure a \emph{regular subproblem}~$(T, \bspc)$
  with~$\bspc_n := [\beta_1, \dots, \beta_n]$. Adding extra conditions
  (vanishing of all~$\coeffex{k+\mu_i}$ for sufficiently small~$k$) we
  obtain a boundary space~$\bspc = \bspc_n + \cdots$ such that~$(T,
  \bspc)$ is semi-regular.
\item If a \emph{particular boundary condition} is desired, it may be
  ``traded'' against one of the~$\beta_i$; if this is not possible, it
  can be ``annexed'' to the extra conditions. After these amendments,
  the subproblem~$(T, \bspc_n)$ is still regular, and~$(T, \bspc)$
  still semi-regular.
\item Next we compute the corresponding \emph{accessible
    space}~$T(\orth{\bspc})$. This space might not
  contain~$C^\omega(I)$, as we saw in Example~\ref{ex:ivp} above when
  we insisted on the conditions~$u(0) = u(1)$.
\item We determine a complement~$\espc$ of~$T(\orth{\bspc})$ as
  \emph{exceptional space} in~$(T, \bspc, \espc)$.
\end{enumerate}

Once these steps are completed, we have a regular generalized boundary
problem~$(T, \bspc, \espc)$ whose \emph{Green's operator}~$G$ can be
computed much in the same way as in
Section~\ref{sec:simple-example}. In detail, we get~$G = \tilde{G} Q$,
where~$\tilde{G}$ is the Green's operator of the regular
subproblem~$(T, \bspc_n)$ and~$Q$ the projector onto~$T(\orth{\bspc})$
along~$\espc$. As we shall see, the operators~$G$ and~$Q$ can be
computed as in the usual setting~\cite{RosenkranzRegensburger2008a}.
Let us first address Step~1 of the above program.

\begin{mylemma}
  Let~$T \in \CC(x)[\der]$ be a nonresonant Fuchsian differential
  operator of order~$n$. Then there exists a fundamental system~$u_1,
  \dots, u_n \in \galg$ of~$T$ and~$n$ coefficient
  functionals~$\beta_1 := \coeffex{\mu_1+k_1}, \dots, \beta_n :=
  \coeffex{\mu_n+k_n}$ ordered as~$k_1+\mu_1 < \cdots < k_n+\mu_n$
  so that~$\beta(u) \in \CC^{n \times n}$ is a lower unitriangular
  matrix.
\end{mylemma}

\begin{proof}
  We start from an arbitrary fundamental system $$u_1 = x^{\mu_1}
  \sum_{k \ge k_1} a_{1,k} x^k, \dots, u_n = x^{\mu_n} \sum_{k \ge
    k_n} a_{k \ge k_n} x^k$$ of the Fuchsian operator~$T$, where we
  take~$\mu_1, \dots, \mu_n$ fracational as before and we may assume
  that~$a_{1,k_1}, \dots, a_{n,k_n} = 1$ so that each fundamental
  solution~$u_i$ has order~$k_i$. (The order of a series~$u = x^\mu
  \sum_{k \ge N} a_k \, x^k$ is defined as the smallest integer~$k$
  such that~$\coeffex{k} \, u \neq 0$.) We order the fundamental
  solutions such that~$k_1+\mu_1 \le \cdots \le k_n+\mu_n$. We can
  always achieve strict inequalities as follows. If~$i<n$ is the first
  place where~$k_i+\mu_i = k_{i+1}+\mu_{i+1}$ we must also have~$\mu_i
  = \mu_{i+1}$ since~$0 \le \mu_i, \mu_{i+1} < 1$. Therefore we
  have~$k_i = k_{i+1}$, and we can replace~$u_{i+1}$ by~$u_{i+1} -
  u_i$ and make it monic so as to ensure~$k_i+\mu_i <
  k_{i+1}+\mu_{i+1}$. Repeating this process at most~$n-1$ times we
  obtain~$k_1+\mu_1 < \cdots < k_n+\mu_n$. Choosing now the boundary
  functionals~$\beta_1 := \coeffex{\mu_1+k_1}, \dots, \beta_n :=
  \coeffex{\mu_n+k_n}$ as in the statement of the lemma, we have
  clearly~$\beta_i(u_i) = 1$ and~$\beta_i(u_j) = 0$ for~$j>i$ as
  claimed.
\end{proof}

In particular we see that~$E_n := (\beta_1, \dots, \beta_n)(u_1,
\dots, u_n)$ has unit determinant, so it is regular. Setting~$\bspc_n
:= [\beta_1, \dots, \beta_n]$, we obtain a regular boundary
problem~$(T, \bspc_n)$. Note that some of the~$\mu_i$ may
coincide. For each~$\mu \in M := \{ \mu_1, \dots, \mu_n\}$ let~$k_\mu$
be the smallest of the~$k_i$ with~$\mu = \mu_i$. We expand the
$n$~boundary functionals by suitable curbing constraints to the
full boundary space
\begin{equation}
  \label{eq:curb-constr}
  \bspc := \bspc_n + [ \coeffex{k+\mu} \mid \mu \in M, \: k<k_\mu ]
\end{equation}
since the inhomogeneous solutions should be at least as smooth (in the
sense of pole order) as the homogeneous ones. Note that~$(T, \bspc)$
is clearly a semi-regular boundary problem. This achieves Step~1 in
our program.

Now for Step~2. Suppose we want to \emph{impose a boundary
  condition}~$\beta$, assuming it is of the type discussed above
(composed of derivatives~$\evl_\xi D^l$ and coefficient
functionals~$\coeffex{k+\mu}$ for fractional parts~$\mu$ of indicial
roots). We must distinguish two cases:
\begin{description}
\item[Trading.] If the row vector~$r := \beta(u_i)_{i = 1, \dots, n}
  \in \CC^n$ is nonzero, we can express it as a $\CC$-linear
  combination~$c_1 r_1 + \cdots + c_n r_n$ of the rows~$r_1, \dots,
  r_n$ of~$E_n$. Let~$k$ be the largest index such that~$c_k \neq
  0$. Then we may express~$r_k$ as a $\CC$-linear combination of~$r$
  and the remaining rows~$r_i \; (i \neq k)$, hence we may
  exchange~$\beta_k$ with~$\beta$ without destroying the regularity
  of~$E_n = (\beta_1, \dots, \beta_n)(u_1, \dots, u_n)$.
\item[Annexation.] Otherwise, we have~$\Ker(T) \le
  \orth{\beta}$. Together with~$\Ker(T) \dirs \orth{\bspc_n} = \galg$
  this implies~$\orth{\smash{\big([\beta] \cap \bspc_n\big)}\!} =
  \orth{\beta} + \orth{\bspc_n} = \galg$ and hence~$\beta \not\in
  \bspc_n$ by the identities
  of~\cite[App.~A]{RegensburgerRosenkranz2009}. Furthermore,
  Lemma~4.14 of~\cite{Korporal2012} yields
  \begin{equation*}
    T( \orth{\bspc_n+\beta} ) = 
    T( \orth{\bspc_n} \cap \orth{\beta} ) = T(\orth{\bspc_n}) \cap
    T(\orth{\beta}) = T(\orth{\beta})
  \end{equation*}
  since we have~$T(\orth{\bspc_n}) = \galg$ from the regularity
  of~$(T, \bspc)$. But this means that adding~$\beta$ to~$\bspc$ as a
  new boundary condition necessarily cuts down the space of accessible
  functions unless~$\beta$ happens to be in the span of the curbing
  constraints~$\coeffex{k+\mu} \in \bspc$ added to~$\bspc_n$
  in~\eqref{eq:curb-constr}.
\end{description}
In the sense of the above discussion (see Example~\ref{ex:ivp}), the
first case signifies a ``natural'' choice of boundary condition while
the second case means we insist on imposing an extra condition (unless
it is a redundant curbing constraint). Repeating these steps as the
cases may be, we can successively impose any (finite) number of given
boundary conditions. This completes Step~2.

For Step~3 we require the computation of the accessible
space~$T(\orth{\bspc})$. We shall now sketch how this can be done
algorithmically, starting with a finitary description of the
\emph{admissible space}~$\orth{\bspc}$ as given in the next
proposition. The proof is unfortunately somewhat tedious and
long-wided but the basic idea is simple enough: We substitute a series
ansatz into the boundary conditions specified in~$\bspc$ to determine
a number of lowest-order coefficients. The rest is just some
bureaucracy for making sure that everything works out (most likely
this could also be done in a more effective way).

\begin{myproposition}
  \label{prop:char-adspc}
  Let~$(T, \bspc)$ be a semi-regular boundary problem of order~$n$
  with mild singularity such that~$(T,\bspc_n)$ is a regular
  subproblem. Let~$M = \{\mu_1, \dots, \mu_n \}$ be the fractional
  parts of the indicial roots for~$T$. Then we have a direct
  decomposition of the admissible space~$\orth{\bspc} = \bigoplus_{\mu
    \in M} x^\mu \aspc_\mu$ with components
  \begin{equation}
    \label{eq:fin-repr}
    \aspc_{\mu} = [p_{\mu 1}(x), \dots, p_{\mu \, l_{\mu}}\!(x)] +
    P_{\mu} \big( C^\omega(I)^M \big) \qquad (\mu \in M).
  \end{equation}
  Here $p_{\mu 1}(x), \dots, p_{\mu \, l_{\mu}}(x) \in \CC[x,
  \tfrac{1}{x}]$ are linearly independent Laurent polynomials and the
  linear operators~$P_{\mu}\colon C^\omega(I)^M \to C^\omega(I)$ are
  defined by
  \begin{equation*}
    P_{\mu}\big( b(x) \big) = x^{j_{\mu}} b_{\mu}(x) + \sum_{\nu
      \in M} \sum_{\xi,j} q_{\mu \nu \xi j}(x) \, \evl_\xi D^j(b_\nu),
  \end{equation*}
  with $j_\mu \in \ZZ$ and Laurent polynomials~$q_{\mu \nu \xi j}(x)
  \in \CC[x, \tfrac{1}{x}]$, almost all of which vanish over the
  summation range~$0 < \xi \le 1$ and~$j \ge 0$.
\end{myproposition}

\begin{proof}
  Any element~$u \in \galg$ may be written in
  the form~$$u = \sum_{\mu \in M} \sum_{k=k(\mu)}^\infty a_{\mu,k} \,
  x^{k+\mu},$$ where the~$k(\mu) \in \ZZ$ are \emph{a priori}
  arbitrary. However, because of the curbing constraints of~$\bspc$
  in~\eqref{eq:curb-constr} we may assume~$k(\mu) = k_\mu$ for~$u \in
  \orth{\bspc}$. Let us write~$\bspc_+$ for the extra conditions that
  were ``annexed'' at Step~2 (if none were added then clearly~$\bspc_+
  = O$); thus~$\bspc$ is a direct sum of~$\bspc_n$ and~$\bspc_+$ and
  the curbing constraints. If~$s_\mu$ denotes the largest integer such
  that~$\coeffex{s+\mu}$ occurs in any condition of~$\bspc_n +
  \bspc_+$ and if we set~$r_\mu := s_\mu+1+\mu$, we can also write~$u
  \in \orth{\bspc}$ as
  \begin{equation}
    \label{eq:ansatz}
    u = \sum_{\mu \in M} \sum_{k=k_\mu}^{s_\mu} a_{\mu k} x^{k+\mu} + 
    \sum_{\mu \in M} x^{r_\mu} \, b_\mu(x),
  \end{equation}
  where the~$b_\mu(x) \in C^\omega(I)$ are convergent power
  series. Next we impose the conditions~$\beta \in \bspc_n + \bspc_+$
  on~$u$. But each~$\beta$ is a $\CC$-linear combintions of
  coefficient functionals~$\coeffex{k+\mu}$ with~$u \mapsto a_{\mu k}$
  and derivatives~$\evl_\xi D^l\: (0 < \xi \le 1, l \ge 0)$ whose
  action on~$u$ yields
  \begin{equation*}
    \sum_\mu \sum_k a_{\mu k} \, \fafac{(k+\mu)}{l} \, \xi^{k+\mu-l} 
    + \sum_\mu \sum_{j=0}^l \binom{l}{j} \fafac{r_\mu}{j} \,
    \xi^{r_\mu-j} \, b_\mu^{(l-j)}(\xi),
  \end{equation*}
  Hence~$\beta(u)$ yields a~$\CC$-linear combination of
  the~$a_{\mu k}$ and the~$b_\mu^{(j)}(\xi)$. We compile the
  coefficients~$a_{\mu k} \in \CC$ into a column vector~$\hat{a} \in
  \CC^S$ of size~$S := \sum_\mu (s_\mu-k_\mu+1)$, consisting of~$|M|$
  contiguous blocks of varying size~$s_\mu-k_\mu+1$. Putting the
  (finite) set~$\Xi$ of evaluation points occurring in~$\bspc_n +
  \bspc_+$ into ascending order, we compile also the
  derivatives~$b_\mu(\xi), b_\mu'(\xi), \dots$ into a column
  vector~$\hat{b} \in \CC^T$ of size~$T := \sum_\mu t_\mu$, consisting
  of~$|M|$ contiguous blocks, each holding~$t_\mu$
  derivatives~$\smash{b_\mu^{(j)}}(\xi)$ for a fixed~$\mu \in M$ and
  certain~$\xi \in \Xi$ and~$j \ge 0$. Assuming~$\bspc_n + \bspc_+$
  has dimension~$R$, the boundary conditions~$\beta(u) = 0$ for~$\beta
  \in \bspc_n + \bspc_+$ can be written as
  \begin{equation}
    \label{eq:bc-eqns}
    \hat{A} \hat{a} = \hat{B} \hat{b}
  \end{equation}
  for suitable matrices~$\hat{A} \in \CC^{R \times S}$ and~$\hat{B}
  \in \CC^{R \times T}$ that can be computed from the boundary
  functionals~$\beta \in \bspc_n + \bspc_+$.

  Regarding the right-hand side as given, let us now
  put~\eqref{eq:bc-eqns} into row echelon form (retaining the same
  letters for simplicity). If the resulting system contains any rows
  that are zero on the left but nonzero on the right, this signals
  constraints amongst the~$b_\mu(x)$ whose treatment is postponed
  until later. For the moment we discard such rows as well rows that
  are zero on both sides. Let~$U \le R$ be the number of the remaining
  rows and~$p_1, \dots, p_U$ the pivot positions and~$V := S-U$ the
  number of free parameters. Then we can solve~\eqref{eq:bc-eqns} in
  the form~$\hat{a} = \tilde{a} + \tilde{A} \cdot
  \CC^V$. Here~$\tilde{a} \in \CC^S$ is the vector with
  entry~$(\hat{B} \hat{b})_j$ in row~$p_j$ and zero otherwise,
  and~$\tilde{A} \in \CC^{S \times V}$ consists of the non-pivot
  columns of the corresponding padded matrix (its $j$-th row is
  the~$i$-th row of~$\hat{A}$ for pivot indices $j=p_i$, and~$-e_j$
  for non-pivot indices~$j$). Writing the free parameters as~$\hat{v}
  = (v_1, \dots, v_V) \in \CC^V$, we may substitute this
  solution~$\hat{a}$ into the ansatz~\eqref{eq:ansatz} to obtain
  \begin{equation}
    \label{eq:new-ansatz}
    u = \sum_{\mu \in M} \sum_{k=k_\mu}^{s_\mu} a_{\mu k}(\hat{v},
    \hat{b}) \, x^{k+\mu} + \sum_{\mu \in M} x^{r_\mu} \, b_\mu(x),
  \end{equation}
  where the~$a_{\mu k}(\hat{v}, \hat{b})$ are~$\CC$-linear
  combinations of the free parameters~$\hat{v}$ and the
  derivatives~$\smash{b_\mu^{(j)}}(\xi)$ comprising~$\hat{b}$.  Note
  that we may regard the~$a_{\mu k}(\dots)$ as row vectors in~$\CC^{1
    \times (V+T)}$ whose entries are computed from~$\hat{A}$
  and~$\hat{B}$.

  We turn now to the issue of constraints amongst the~$b_\mu(x)$,
  embodied in those rows of the reduced row echelon form
  of~\eqref{eq:bc-eqns} that are zero on the left-hand side but
  nonzero on the right-hand side. We put the corresponding block
  of~$\hat{B}$ into reduced row echelon form (again retaining the same
  letters for simplicity). Let~$\mu'$ the smallest~$\mu$ occurring in
  any such row. Then each of the rows containing a
  derivative~$\smash{b_{\mu'}^{(j)}}(x)$ provides a constraint of the
  form
  \begin{equation}
    \label{eq:row-constr}
    \sum_{\xi \in \Xi} \sum_j c_{\xi j} \, b_{\mu'}^{(j)}(\xi) + \sum_{\xi \in \Xi}
    \sum_{\mu > \mu'} \sum_j c_{\xi\mu j} \, b_{\mu}^{(j)}(\xi) = 0,
  \end{equation}
  where the~$c_{\xi j}, c_{\xi\mu j} \in \CC$ are determined
  by~$\hat{B}$, and with finite sums over~$j$. Let the number of such
  constraints be~$X$. Collecting
  the~$\smash{\big(b_{\mu'}^{(j)}(\xi)\big)}_{\xi,l}$ into a
  vector~$\hat{b}_{\mu'} \in \CC^Y$ of size~$Y := t_{\mu'}$, we can
  write the constraints~\eqref{eq:row-constr} as matrix
  equation~$\hat{C} \hat{b}_{\mu'} = \hat{d}$ where the coefficient
  matrix~$\hat{C} \in \CC^{X \times Y}$ is determined by the~$c_{\xi
    j}$ of each constraint~\eqref{eq:row-constr}, and the right-hand
  side~$\hat{d} \in \CC^X$ by the corresponding~$c_{\xi\mu j}$ and
  the~$\smash{b_\mu^{(j)}}(\xi)$ for~$\mu>\mu'$, which for the moment
  we regard as known. Note that~$X<Y$ since~$\hat{C}$ is in row
  echelon form. Let its pivots be in the positions~$q_1, \dots, q_X$
  and set~$Z := Y-X$. Then we can write the solution
  as~$\hat{b}_{\mu'} = \tilde{b}_{\mu'} + \tilde{C} \cdot
  \CC^Z$. Here~$\tilde{b}_{\mu'} \in \CC^Y$ is the vector with
  entry~$\hat{d}_j$ in row~$q_j$ and zero otherwise, while~$\tilde{C}
  \in \CC^{Y \times Z}$ consists of the non-pivot columns of the
  padded matrix (its $j$-th row is the $i$-th row of~$\hat{C}$ for
  pivot indices $j=q_i$, and~$-e_j$ for non-pivot
  indices~$j$). Writing~$\hat{w} = (w_1, \dots, w_Z) \in \CC^Z$ for
  the corresponding free parameters, we may thus
  view~\eqref{eq:row-constr} as providing~$Y$ constraints
  \begin{equation}
    \label{eq:det-eqns}
    b_{\mu'}^{(j)}(\xi) = b_{\mu' \xi j}(\hat{w}, \hat{b}_+),
  \end{equation}
  where the~$b_{\mu'\xi j}(\hat{w}, \hat{b}_+)$ are~$\CC$-linear
  combinations of the free parameters~$\hat{w}$ and certain
  derivatives~$\smash{b_\mu^{(j)}}(\xi)$ that we have collected into a
  vector~$\hat{b}_+ \in \CC^{T-Y}$. Again we may view~$b_{\mu' \xi
    j}(\dots)$ as a row vector in~$\CC^{1 \times (Z+T-Y)}$ whose
  entries can ultimately be computed from~$\hat{A}$ and~$\hat{B}$.

  We regard now~\eqref{eq:det-eqns} as determining equations for
  fixing~$Y$ of the coefficients of~$b_{\mu'}(x) = \sum b_{\mu' k}
  x^k$. Indeed, if~$j$ is the highest derivative order occurring
  in~\eqref{eq:det-eqns} we may split according to~$b_{\mu'}(x) =
  \sum_{k \le j} b_{\mu' k} x^k + x^{j+1} \bar{b}_{\mu'}(x)$, with an
  arbitrary power series~$\bar{b}_{\mu'}(x)$, and then substitute this
  into~\eqref{eq:det-eqns} to obtain
  \begin{equation}
    \label{eq:coeff-det}
    b_{\mu' j} = b_{\mu' \xi j}(\hat{w}, \hat{b}_+) - \sum_{i=0}^j
    \binom{j}{i} \, \frac{j+1}{i+1} \, \xi^{i+1}
    \, \bar{b}_{\mu'}^{(i)}(\xi)
  \end{equation}
  for fixing one coefficient of~$b_{\mu'}(x)$. Now we substitute
  this~$b_j$ back into the above ansatz~$b_{\mu'}(x) = \sum_{k \le j}
  b_{\mu' k} x^k + x^{j+1} \bar{b}_{\mu'}(x)$, and we repeat the whole
  process for all other constraints~\eqref{eq:det-eqns}, each time
  determining the lowest unknown coefficient~$b_k$
  of~$b_{\mu'}(x)$. Of course, it may be necessary to expand the
  splitting to extract a larger polynomial part (if all the
  coefficients in the current polynomial part are
  determined). Eventually, we end up with
  \begin{equation}
    \label{eq:res-series}
    b_{\mu'}(x) = \sum_{k \le m_{\mu'}} b_{\mu' k}(\hat{w}, \hat{b}_+)
    \, x^k + x^{m_{\mu'}+1} \, \bar{b}_{\mu'}(x),
  \end{equation}
  with some break-off index~$m_{\mu'}$ that is specific
  to~$b_{\mu'}(x)$, and with~$\hat{b}_+$ enlarged to comprise also the
  derivatives~$\smash{\bar{b}_{\mu'}^{(j)}}(\xi)$ that were needed in
  determining the
  coefficients~\eqref{eq:coeff-det}. Substituting~\eqref{eq:res-series}
  into~\eqref{eq:new-ansatz}, we see that the~$b_{\mu' k}(\hat{w},
  \hat{b}_+)$ may be combined with the~$a_{\mu' k}(\hat{v}, \hat{b})$
  if we adjoin the parameters~$\hat{w}$ to the parameters~$\hat{v}$;
  then we can also rename the series~$\bar{b}_{\mu'}(x)$ back
  to~$b_{\mu'}(x)$. Of course new terms~$a_{\mu' k}(\hat{v}, \hat{b})
  \, x^{k+\mu'}$ may be created in the ansatz~\eqref{eq:new-ansatz},
  and its polynomial part may be expanded---but its overall form is
  not altered.

  We have now eliminated those constraints amongst the~$b_{\mu}(x)$
  occurring in~\eqref{eq:bc-eqns} that involve~$b_{\mu'}$,
  where~$\mu'$ was chosen minimal. Hence we are only left with
  constraints amongst the~$b_{\mu}(x)$ with~$\mu > \mu'$. Repeating
  the elimination process a finite number of times (at most~$|M|$
  eliminations are necessary), we obtain the generic form of~$u \in
  \orth{\bspc}$ as given in~\eqref{eq:new-ansatz}, where
  the~$b_\mu(x)$ are now arbitrary (convergent) power series. Since
  the terms with distinct factors~$x^\mu$ are clearly in distinct
  direct sum components~$x^\mu \aspc_\mu$, the latter may now be
  described by
  \begin{equation*}
    \aspc_\mu = \bigg\{ \sum_{k=k_\mu}^{s_\mu} a_{\mu k}(\hat{v},
    \hat{b}) \, x^k + x^{j_\mu} \, b_\mu(x) \biggm| 
    \double{\hat{v} \in \CC^V,}{b(x) \in C^\omega(I)^M} \bigg\} ,
  \end{equation*}
  where we have set~$j_\mu := s_\mu + 1$. Splitting the
  coefficients~$a_{\mu k}(\hat{v}, \hat{b})$ into a~$\CC$-linear
  combination of the free parameters~$(v_1, \dots, v_V)$ and a
  $\CC$-linear combination of the
  derivatives~$\smash{b_\nu^{(j)}}(\xi)$ comprising~$\hat{b}$, we
  collect terms whose coefficient is a specific parameter~$v_i \: (i =
  1, \dots, V)$ or a specific derivative~$\smash{b_\nu^{(j)}}(\xi) \:
  (\nu \in M, \xi \in \Xi, j \ge 0)$. This leads to
  \begin{equation*}
    \aspc_\mu = \bigg\{ \sum_{i=1}^V v_i \, p_{\mu i}(x) + x^{j_\mu}
    \, b_\mu(x) + \sum_{\nu \in M} \sum_{\xi, j} b_{\nu}^{(j)}(\xi) \,
    q_{\mu\nu\xi j}(x) \biggm| \double{\hat{v} 
    \in \CC^V,}{b(x) \in C^\omega(I)^M} \bigg\}
  \end{equation*}
  and hence to~\eqref{eq:fin-repr} by extracting a $\CC$-basis for
  each~$[p_{\mu 1}(x), \dots, p_{\mu V}(x)]$.
\end{proof}

We have now established Step~3 of our program since the accessible
space~$T(\orth{B})$ can be specified by applying~$T \in \CC[x,
\tfrac{1}{x}$ to the generic functions~\eqref{eq:fin-repr} of the
admissible space~$\orth{B}$. Our next goal is to find a
\emph{projector}~$Q$ onto~$T(\orth{B})$, which then gives the
exceptional space~$\espc := \Ker(Q) = \Img(1-Q)$ required for
Step~4. This will be easy once we have a corresponding projector
onto~$\orth{\bspc}$. In fact, the operator~$P_\mu$ in
Proposition~\ref{prop:char-adspc} is not quite a projector (for one
thing, it is not even an endomorphism), but in a sense it is not far
away from being one. For seeing this, note that we have
\begin{equation}
  \label{eq:dir-decomp}
  \galg = \bigoplus_{\mu \in M} \, x^\mu \, \CC((x)),
\end{equation}
where~$\CC((x))$ denotes the field of Laurent series (converging in
the punctured unit disk). The direct
decomposition~\eqref{eq:dir-decomp} just reflects the fact that
the~$x^\mu$ for distinct~$\mu \in M$ are linearly independent. Let us
write~$\indproj{\mu}\colon \galg \to \galg$ for the \emph{indicial
  projector} onto the component~$x^\mu \, \CC((x))$
of~\eqref{eq:dir-decomp}. In other words, $\indproj{\mu}$ extracts all
terms of the form~$x^{k+\mu} \: (k \in \ZZ)$ from a series
in~$\galg$. Note that combinations like~$\beta=\evl_\xi D^k
\indproj{\mu}$ provide linear functionals~$\beta \in \galg^*$ for
extracting derivatives of the~$\mu$-component of a given series
in~$\galg$.

For writing the projector corresponding to
Proposition~\ref{prop:char-adspc}, let us also introduce the
\emph{auxiliary operator}~$x^\mu\colon \CC((x)) \to x^\mu \CC((x)) \le
\galg$ and its inverse~$x^{-\mu}$. Then we shall see that the required
projector is essentially a ``twisted'' version of two kinds of
projector: one for splitting off the polynomial part of the occurring
series, and one for imposing the derivative terms. For convenience, we
shall use \emph{orthogonal projectors} in~$\CC((x))$, where the
underlying inner product is defined by~$\inner{x^k}{x^l} =
\delta_{kl}$ for all~$k,l \in \ZZ$. Such projectors are always
straightforward to compute (using linear algebra on complex matrices).

\begin{corollary}
  \label{cor:proj-adspc}
  Using the same notation as in Proposition~\ref{prop:char-adspc}, let
  $R_\mu, S_\mu\colon \CC(x)) \to \CC(x)$ be the orthogonal projectors
  onto~$[p_{\mu 1}(x), \dots, p_{\mu \, l_{\mu}}\!(x)]$ and
  onto~$x^{j_\mu} \, \CC((x))$, respectively. Writing~$R_\mu' := x^\mu
  R_\mu x^{-\mu} \indproj{\mu}$ and~$S_\mu' := x^\mu S_\mu x^{-\mu}
  \indproj{\mu}$ for their twisted analogs, we define the linear
  operator~$P\colon \galg \to \galg$ by
  \begin{equation}
    \label{eq:adm-proj}
    P = \sum_{\mu \in M} \Big( R_\mu' + S_\mu' + \sum_{\nu \in
      M} \sum_{\xi,j} x^\mu \, q_{\mu\nu\xi j}(x) \, \evl_\xi D^j
    x^{-j_\mu-\mu} S_\nu' \Big)
  \end{equation}
  is a projector onto~$\orth{\bspc}$.
\end{corollary}

\begin{proof}
  The action of~$P$ on a general series (which we may split at $s_\mu
  = j_\mu - 1$ in its~$\mu$-components)
  \begin{equation*}
    u(x) = \sum_{\mu \in M} x^\mu \, f_\mu(x) = \sum_{\mu \in M}
    \sum_{k=k_\mu}^{s_\mu} f_{\mu k} x^{k+\mu} + \sum_{\mu \in M}
    x^{j_\mu+\mu} \, b_\mu(x) \in \galg
  \end{equation*}
  is
  \begin{equation*}
    Pu(x) = \sum_{\mu \in M} x^\mu \bigg( \! R_\mu \Big( \sum_{k=k_\mu}^{s_\mu}
    f_{\mu k} x^k \Big) + x^{j_\mu} \, b_\mu(x)
    + \sum_{\nu \in M} \sum_{\xi, j} q_{\mu\nu\xi j}(x) \,
    b_\nu^{(j)}(\xi) \bigg),
  \end{equation*}
  where we have set~$b_\mu(x) := x^{-j_\mu} f_\mu(x) \in
  C^\omega(I)$. Extracting the $x^\mu$ component of~$Pu$ yields
  \begin{equation}
    \label{eq:proj-comp}
    R_\mu \Big( \sum_{k=k_\mu}^{s_\mu}
    f_{\mu k} x^k \Big) + x^{j_\mu} \, b_\mu(x)
    + \sum_{\nu \in M} \sum_{\xi, j} q_{\mu\nu\xi j}(x) \,
    b_\nu^{(j)}(\xi) \in \aspc_\mu
  \end{equation}
  as one sees by comparing with the last displayed equation in the
  proof of Proposition~\ref{prop:char-adspc}. Hence we may conclude
  that~$\Img(P) \le \orth{\bspc}$.

  It remains to prove that~$Pu = u$ for~$u \in \orth{\bspc}$ since
  this implies~$\Img(P) \ge \orth{\bspc}$ and~$P^2 = P$ so that~$P$ is
  indeed a projector onto~$\orth{\bspc}$ as claimed in the
  corollary. So assume~$u(x) \in \orth{\bspc}$ is arbitrary, and split
  it as above. The orders~$k_\mu$ of the series~$f_\mu$ are given by
  the curbing constraints. Since~$u(x)$ already satisfies all boundary
  conditions in~$\bspc_n + \bspc_+ \le \bspc$, the reduced row echelon
  form of~\eqref{eq:bc-eqns} will contain only zero rows. The original
  ansatz~\eqref{eq:ansatz} is then left intact; no expansion of the
  polynomial part is necessary and no~$\hat{b}$ are involved in its
  coefficients. Therefore the original coefficients~$a_{\mu k}$
  in~\eqref{eq:ansatz} coincide with the coeffients~$a_{\mu
    k}(\hat{v}, \hat{b}) = a_{\mu k}(\hat{v})$ in~\eqref{eq:bc-eqns},
  which constitute the polynomials of~$\mathcal{P}_\mu := [p_{\mu
    1}(x), \dots, p_{\mu \, l_{\mu}}\!(x)]$. Now let us
  consider~\eqref{eq:proj-comp}. Since the series~$\sum_k f_{\mu k}
  x^k$ is thus already in~$\mathcal{P}_\mu = \Img(R_\mu)$, we may omit
  the action of~$R_\mu$ and since~$a_{\mu k}(\hat{v}, \hat{b}) =
  a_{\mu k}(\hat{v})$ the triple sum in~\eqref{eq:proj-comp} is
  zero. But then~$Pu(x)$ becomes identical with~$u(x)$ as was claimed.
\end{proof}

For accomplishing Step~4 of our program, it only remains to determine
a projector~$Q$ onto the accessible space~$T(\orth{\bspc})$ from the
projector~$P$ onto the admissible space~$\orth{\bspc}$ provided in
Corollory~\ref{cor:proj-adspc}. This can be done easily since~$Q$ is
\emph{essentially a conjugate} of~$P$ except that we use a fundamental
right inverse~$\fri{T}$, for want of a proper inverse. (The formula
for~$\fri{T}$ in~\cite[Prop.~23]{RosenkranzRegensburger2008a}
and~\cite[Thm.~20]{RosenkranzRegensburgerTecBuchberger2012} may be
used but recall that in our case~$\cum = \cum_1^x$ so that~$\evl =
\evl_1$.)

\begin{proposition}
  \label{prop:proj-acspc}
  Using the same notation as in Proposition~\ref{cor:proj-adspc}, the
  operator~$$Q := T P \fri{T}\colon \galg \to \galg$$ is a projector
  onto the accessible space~$T(\orth{\bspc})$.
\end{proposition}

\begin{proof}
  Let us first check that~$Q$ is a projector. Writing~$U := 1 -
  \fri{T} T$ for the projector onto~$\Ker{T}$ along~$[\evl, \evl D,
  \dots, \evl D^{n-1}]$ and using~$P^2 = P$, we will indeed get
  \begin{equation*}
    Q^2 = TP^2 \fri{T} - TPUP\fri{T} = TP\fri{T} = Q
  \end{equation*}
  provided we can ascertain that~$\Ker{T} \le \Ker{P} =: \ospc$ since
  then~$PU=0$. We know that~$\orth{B} \dirs \ospc = \galg$ since~$P$
  is a projector onto~$\orth{\bspc}$ along~$\ospc$. On the other hand,
  we have also~$\Ker(T) \dirs \orth{\bspc_n} = \galg$ since~$(T,
  \bspc_n)$ is regular. Intersecting~$\ospc$ in the former
  decomposition with the latter yields
  \begin{equation}
    \label{eq:decomp1}
    \orth{\bspc} \dirs (\ospc \cap \orth{\bspc_n}) \dirs (\ospc \cap
    \Ker{T}) = \galg.
  \end{equation}
  But~$\bspc_n \le \bspc$ implies~$\orth{\bspc} \le \orth{\bspc_n}$,
  so intersecting the decomposition~$\orth{\bspc} \dirs \ospc = \galg$
  with~$\orth{\bspc_n}$ leads to~$\orth{\bspc} \dirs (\ospc \cap
  \orth{\bspc_n}) = \orth{\bspc_n}$. Using the other
  decompostion~$\Ker(T) \dirs \orth{\bspc_n} = \galg$ one more time we
  obtain
  \begin{equation}
    \label{eq:decomp2}
    \orth{\bspc} \dirs (\ospc \cap \orth{\bspc_n}) \dirs \Ker{T} =
    \galg.
  \end{equation}
  Comparing~\eqref{eq:decomp1} and~\eqref{eq:decomp2}, we can apply
  the well-known rule~\cite[(2.6)]{Korporal2012} to obtain the
  identity~$\ospc \cap \Ker{T} = \Ker{T}$ and hence the required
  inclusion~$\Ker{T} \le \ospc$.

  It remains to prove that~$\Img(Q) = T(\orth{\bspc})$. The inclusion
  from left to right is obivous, so assume~$f=Tu$ with~$u \in
  \orth{\bspc}$. Then~$\fri{T} f = u - Uu$ and hence~$P \fri{T} f = Pu
  - PU u = u$ because~$P$ projects onto~$\orth{\bspc}$ and~$PU=0$ from
  the above. But then we have also~$Qf=TP\fri{T}f = Tu = f$ and in
  particular~$f \in \Img(Q)$.
\end{proof}

We have now sketched how to carry out the four main steps in our
program aimed at the algorithmic treatment of finding/imposing
``good'' boundary conditions on a Fuchsian differential equation with
one (mild) singularity. At the moment we do not have a full
implementation of the underlying algorithms in \tma\ (or any other
system). However, we have implemented a prototype version of some
portion of this theory. We shall demonstrate some of its features by
with example from engineering mechanics.

\section{Application to Functionally Graded Kirchhoff Plates}
\label{sec:appl}

Circular plates play an important role for many application areas in
engineering mechanics and mathematical physics. If the plates are thin
(the ratio of thickness to diameter is small enough), one may employ
the well-known Kirchhoff-Love plate theory~\cite{Reddy2007}, whose
mathematical description is essentially two-dimensional (via a linear
second-order partial differential equation in two independent
variables). We will furthermore restrict ourselves to \emph{circular
  Kirchhoff plates} so as to have a one-dimensional mathematical
model, via a linear ordinary differential equation of second order.

However, we shall not assume homogeneous plates. Indeed, the precise
manufacture of \emph{functionally graded} materials is an important
branch in engineering mechanics. In the case of Kirchhoff plates, the
functional grading is essentially the variable thickness~$t=t(r)$ or
variable bending rigidity~$D=D(r)$ of the plate along its radial
profile. (We write~$r$ for the radius variable ranging between zero at
the center and the outer radius~$r=b$.)

Let~$w=w(r)$ be the displacement of the plate as a function of its
radius. This is the quantity that we try to determine. It induces the
\emph{radial and tangential moments} given, respectively by
\begin{align*}
  M_r(r) &= -D(r) \, \big( w''(r)+ \tfrac{\nu}{r} \, w'(r) \big),\\
  M_\theta(r) &= -D(r) \, \big( \tfrac{1}{r} \, w'(r) + \nu \, w''(r) \big),
\end{align*}
where~$\nu$ is the Poisson's ratio of the plate (which we assume
constant). For typical materials, $\nu$ may be taken as~$\tfrac{1}{3},
\tfrac{1}{4}, \tfrac{1}{5}$ or even~$0$. A reasonable constitutive law
for the bending rigidity is
\begin{equation}
  \label{eq:bend-rig}
  D(r) = \tfrac{E(r) \, t(r)^3}{12(1-\nu^2)},
\end{equation}
where~$E=E(r)$ is the variable Young's modulus of the plate.

The \emph{equilibrium equation} can then be written as
\begin{equation}
  \label{eq:equi-eqns}
  \frac{d M_r}{dr} + \frac{M_r-M_\theta}{r} = Q_r,
\end{equation}
where~$Q_r = Q_r(r)$ is the cumulative load
\begin{equation*}
  Q_r(r) = -\frac{1}{2\pi r} \int_0^r q(r) \, 2\pi r \, dr =
  -\frac{1}{r} \int_0^r q(r) \, r \, dr
\end{equation*}
induced by a certain \emph{loading}~$q=q(r)$ that may be thought to
describe the weight (or other forces) acting in each ring~$[r, r+dr]$.

For the calculational treatment of~\eqref{eq:equi-eqns} it it useful
to introduce the function~$\phi := - w'(r)$, which represents the
(negative) slope of the plate profile. In terms of~$\phi$, the
equilibrium equation is given by
\begin{equation*}
  \phi''(r) + \Big( \frac{1}{r} + \frac{D'(r)}{D(r)} \Big) \, \phi'(r)
  + \Big( \nu \, \frac{D'(r)}{D(r)} - \frac{1}{r}
  \Big) \frac{\phi(r)}{r} = \frac{Q_r(r)}{D(r)}.
\end{equation*}
A typical example of a useful thickness grading is the \emph{linear
  ansatz}~$t = t_0(1-\tfrac{r}{b})$, cut off beyond some~$a<b$ close
to~$b$; this describes a radially symmetric pointed plate with
straight edges (more or less a very flat cone). Suppressing the
cut-off for the moment and changing the independent variable~$r$
to~$\rho = r/b$, we have thus thickness~$t(\rho) = t_0 \cdot (1-\rho)$
and from~\eqref{eq:bend-rig} bending rigidity~$D(\rho) = D_0 \cdot
(1-\rho)^3$ with~$D_0 := E t_0^3 / 12(1-\nu^2)$. The equilibrium
equation becomes now
\begin{equation*}
  \phi''(\rho) + \Big( \frac{1}{\rho} - \frac{3}{1-\rho} \Big)
  \phi'(\rho) - \Big( \frac{1}{\rho} + \frac{1}{1-\rho} \Big)
  \frac{\phi(\rho)}{\rho} = \frac{Q_r(\rho) \, b^2}{D_0 (1-\rho)^3},
\end{equation*}
where we have set~$\nu=\tfrac{1}{3}$ for simplicity. Note that the
right-hand side of this equation, which we shall designate
by~$f(\rho)$, is not a fixed function of~$\rho$ but depends on our
choice of the loading. Hence we consider~$f(x)$ as a \emph{forcing
function}.

For the \emph{boundary conditions} (for once we use this word in its
original literal sense!) we shall use~$w'(0) = w'(a) = 0$, which
translates to~$\phi'(0) = \phi'(\beta) = 0$ in the~$\phi=\phi(\rho)$
formulation, with the abbreviation~$\beta := a/b < 1$. Physically
speaking, this corresponds to a plate that is clamped in the center
and left free at its periphery (this comes from translating suitable
boundary conditions for the displacement~$w=w(r)$ and taking the
appropriate limits). In summary, we have the boundary problem
\begin{equation}
  \label{eq:kirchhoff-bvp}
  \bvp{  \phi''(\rho) + \Big( \frac{1}{\rho} - \frac{3}{1-\rho} \Big)
    \phi'(\rho) - \Big( \frac{1}{\rho} + \frac{1}{1-\rho} \Big)
    \frac{\phi(\rho)}{\rho} = f(\rho),\vspace*{0.5ex}}{\phi(0) = \phi(\beta) = 0,}
\end{equation}
which is indeed of the type discussed in
Section~\ref{sec:bp-one-sing}, by a simple scaling from~$I=[0,1]$
to~$[0,\beta]$. Its treatment in the \grgr\ package proceeds as
follows [we apologize for the lousy graphics rendering---we plan to
fix this as soon as possible]:

\medskip
\noindent\includegraphics[width=\textwidth,height=0.4\textwidth]{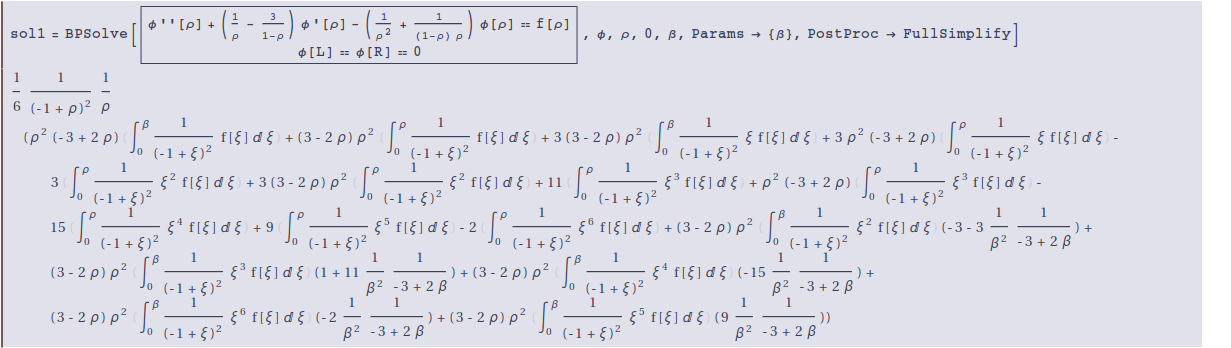}
\medskip

The output is the Green's operator~$G$ of~\eqref{eq:kirchhoff-bvp}
applied to a generic forcing function~$f(\rho)$, giving the
solution~$\phi(\rho) = Gf\,(\rho)$ as an integral
\begin{equation*}
  \int_0^1 g(\rho, \xi) \, f(\xi) \, d\xi
\end{equation*}
in terms of the Green's function~$g(\rho, \xi)$; the latter can also
be retrieved explicitly if this is desired (note that the expression
is clipped on the right-hand side so the two case conditions~$\xi \le
\rho$ and~$\rho < \xi$ labelling the two lines are not visible):

\medskip
\noindent\includegraphics[width=\textwidth,height=0.3\textwidth,%
trim = 0 0 1100 0,clip]{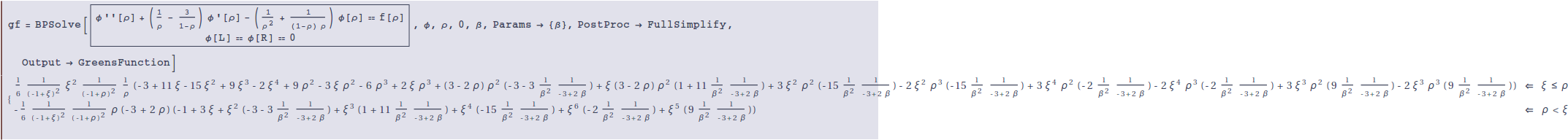}
\medskip

Let us now choose a constant loading~$q(\rho) = q_0$. Taking the
cut-off into account, this leads to the forcing function~$f(\rho) =
\tfrac{\rho +1}{\rho (\rho -1)^2}$. In this case, we can compute the
solution~$\phi(\rho)$ of~\eqref{eq:kirchhoff-bvp} explicitly. For
definiteness, let us choose a cut-off at~$\beta = 0.9$. In that case,
we obtain
\begin{align*}
  \phi(&\rho) = \Big( -2790 \rho ^3+1944 \, \rho ^3 \log
  (9/\rho-9)+2000 \, \rho ^3 \log (1-\rho )\\
  & -2000 \rho ^3 \log (10-10\rho)+4671 \rho ^2-2916 \rho ^2
  \log (9/\rho-9) -3000 \rho ^2 \log (1-\rho )\\
  & +3000 \rho ^2 \log (10-10\rho)-1944 \rho +972 \log (1-\rho )
  \Big) \Big/
  \Big( 2916 (\rho -1)^2 \rho \Big),
\end{align*}
and the corresponding displacement~$w(\rho) = -\cum_0^\beta\, \phi(\rho)
\, d\rho$ is given by
\begin{align*}
  w(&\rho) = \Big( -972 (\rho -1) \,
  \text{Li}_2\left((1-\rho)^{-1}\right)+972 (\rho -1) \,
  \text{Li}_2(1-\rho )-2790 \rho ^2\\
  & +1944 \rho ^2 \log (9) +3944 \rho ^2 \log (1-\rho )-2000 \rho ^2
  \log(-10 (\rho -1))\\
  & -1944 \rho ^2 \log (\rho )+2853 \rho +500 \rho \log ^2(10)+14 \rho
  \log ^2(1-\rho )\\
  & -500 \rho \log ^2(-10 (\rho -1))-14 \log ^2(1-\rho
  )+500 \log ^2(-10 (\rho -1))\\
  & -972 \rho \log (9)-1500 \rho \log (100)+486 \rho \log (81) \log
  (1-\rho )-7825 \rho \log (1-\rho )\\
  & +4000 \rho \log (-10 (\rho
  -1))+972 \rho \log (1-\rho ) \log (\rho )+972 \rho \log (\rho )\\
  & -1944
  \log \left(9/\rho-9\right)-486 \log (81) \log (1-\rho
  )+6825 \log (1-\rho )\\
  & -3000 \log (-10 (\rho -1))-972 \log (1-\rho )
  \log (\rho )-1944 \log (\rho )-500 \log ^2(10)\\
  & +1944 \log (9)+1500\log (100) \Big)
  \Big/ 2916 (1 - \rho)
\end{align*}

We have displayed the graphs of these solutions~$\phi(\rho)$
and~$w(\rho)$ below.

\medskip
\noindent
\includegraphics[width=0.5\textwidth,height=0.5\textwidth]{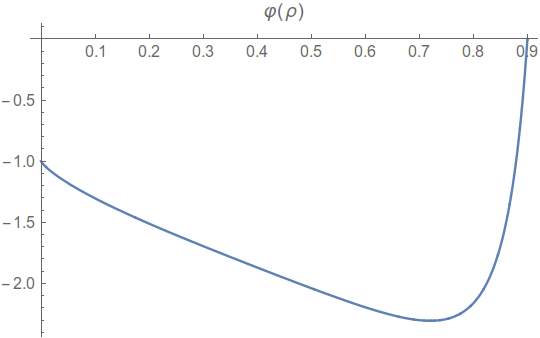}
\includegraphics[width=0.5\textwidth,height=0.5\textwidth]{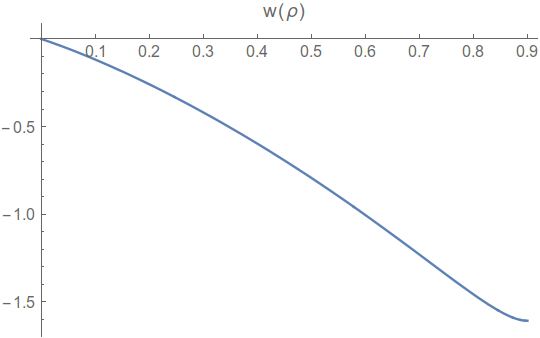}
\medskip

Of course one could use different functional gradings and/or loading
functions, and the integrals would not always come out in closed
form. In this case one can resort to numerical integration (which is
also supported by \emph{Mathematica}).


\end{document}